\documentclass[12pt,a4paper]{article}

\usepackage{fancyhdr,amsmath, amssymb, amsfonts, dsfont, amsthm,tikz,color,epsfig,graphicx,soul}

\newlength{\guillotine}
	\newtheorem{thm}{Theorem}[section]
	
	\newtheorem{lemma}[thm]{Lemma}

	\newtheorem{definition}[thm]{Definition}

	\theoremstyle{remark}
	\newtheorem{rem}[thm]{Remark}
	\newcommand{\RR}{\mathbb R}
	\newcommand{\NN}{\mathbb N}

	\newcommand{\calg}{\mathcal G}
	\newcommand{\calv}{\mathcal V}

    \newcommand{\alf}{\alpha}
	\newcommand{\de}{\delta}
	\newcommand{\la}{\lambda}
	\newcommand{\De}{\Delta}

	\makeatletter
	\newcommand{\ostar}{\mathbin{\mathpalette\make@circled\star}}
	\newcommand{\make@circled}[2]{%
	  \ooalign{$\m@th#1\smallbigcirc{#1}$\cr\hidewidth$\m@th#1#2$\hidewidth\cr}%
	}
	\newcommand{\smallbigcirc}[1]{%
	  \vcenter{\hbox{\scalebox{0.77778}{$\m@th#1\bigcirc$}}}%
	}
	\makeatother

    \newcommand{\un}{\underline}
	\DeclareMathOperator{\Jac}{Jac}
	\DeclareMathOperator{\Leb}{Leb}
	\DeclareMathOperator{\area}{area}
	\DeclareMathOperator{\diam}{diam}

\begin{document}

\title{An upper bound on the dimension of the Rauzy gasket}
\author{Mark Pollicott\thanks{Supported by ERC grant 833802-resonances and EPSRC grant EP/T001674/1.}{$\;$}
and Benedict Sewell\thanks{Supported by the Alfr\'ed R\'enyi Institute Young Researcher Fund.}}

\maketitle

\section{Introduction}

The Rauzy Gasket $\calg$ is a compact subset of the standard 2-simplex,
	$
		    \De
		=
		    \{
		        (x, y, z)
		    \;:\;
		        x, y, z \geq 0,\ %
		        x + y + z = 1
		      \}.
	$
It plays the role of an exceptional set in the theory of interval exchange transformations and other settings, and is the limit set of the iterated function scheme for the three weak projectivised linear maps $T_1, T_2, T_3: \De \to \De$, defined by
	\begin{gather*}
		 	T_1(x,y,z)
		=
		 	\left(
		  		\frac{1}{2-x},
		  		\frac{y}{2-x},
		  		\frac{z}{2-x}
		  	\right),
	\\
		  	T_2(x,y,z)
		=
			  \left(
				  \frac{x}{2-y},
				  \frac{1}{2-y},
				  \frac{z}{2-y}
			  \right),
	\\
			T_3(x,y,z)
		=
	  		\left(
	  			\frac{x}{2-z},
	  			\frac{y}{2-z},
	  			\frac{1}{2-z}
	  		\right);
	\end{gather*}
i.e., $\calg$ is the smallest non-trivial closed set such that $\calg = \bigcup_{j=1}^3 T_j (\calg)$.

 	\begin{figure}
 		\centerline{\includegraphics[height=5.8cm]{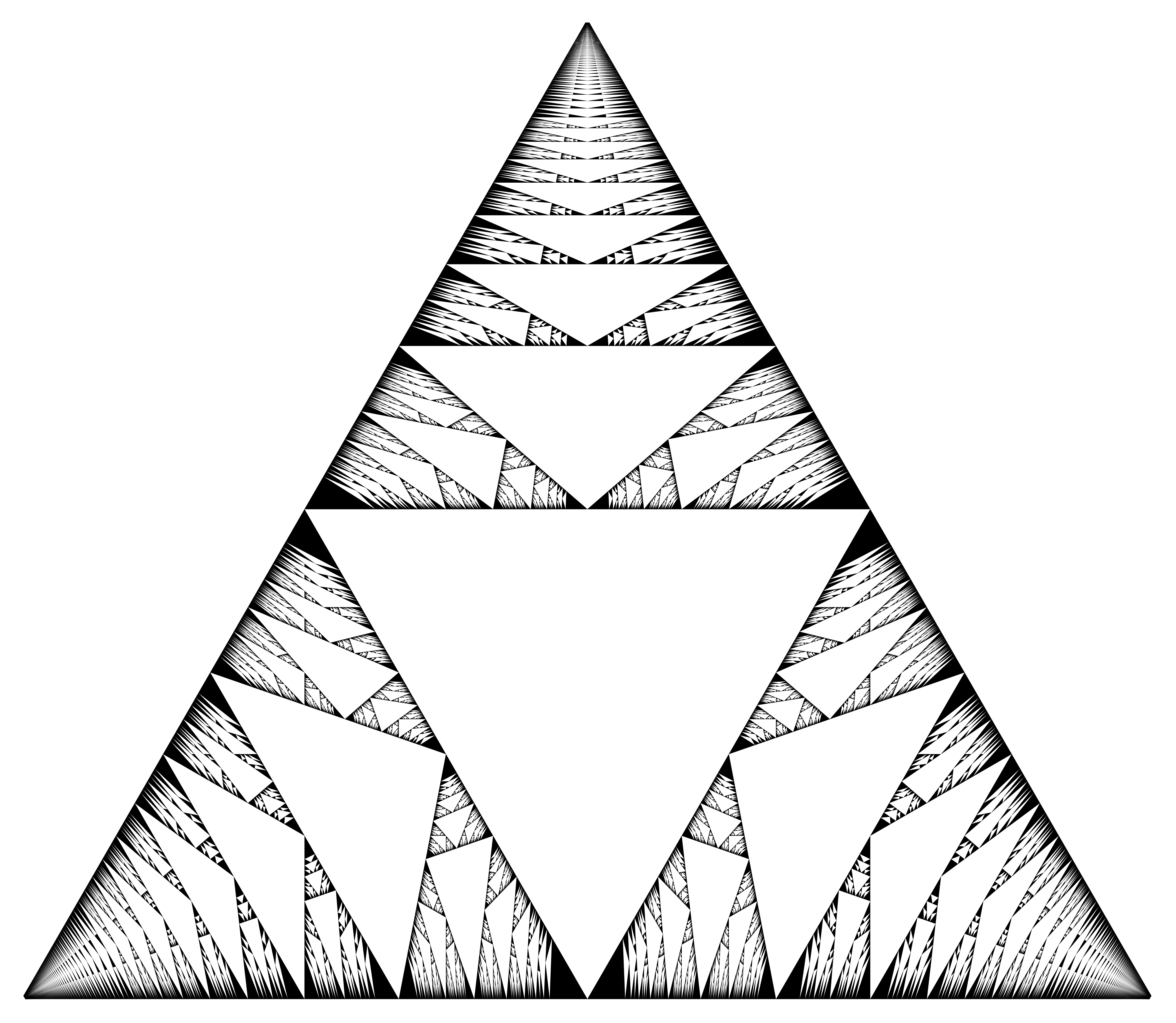}}
 		\caption{The Rauzy gasket}
	\end{figure}
\smallskip
The gasket has an interesting history, appearing for the first time in 1991 in the work of Arnoux and Rauzy \cite{ar}, in the context of interval exchange transformations, where it was conjectured that $\Leb(\calg)=0$.
The gasket was rediscovered by Levitt in 1993 \cite{lev}, in a paper which also included a proof (due to Yoccoz) that  $\Leb(\calg)=0$.
The gasket $\calg$ emerged for a third time in the work of De Leo and Dynnikov \cite{dd}, this time in the context of Novikov's theory of magnetic induction on monocrystals (see \cite{rauzy dynamics} for the dichotomy between this and \cite{ar}). They gave an alternative proof that $\Leb(\calg)=0$ and proposed the stronger result $\dim_H(\calg)<2$.
Novikov and Maltsev \cite{mn} also conjectured the stronger bound $\dim_H(\calg) < 2$, which was rigorously established by Avila, Hubert and Skripchenko \cite{ahs}.
Empirical estimates in \cite{dd} suggest $\dim_H(\calg) \approx 1.72$, and a lower bound was shown in \cite{G-R-Matheus}. Lastly, Fougeron used semiflows and thermodynamic techniques to show $\dim_H(\calg) < 1.825$ \cite{fou}.   Using completely elementary methods, we show the following improved upper bound.

\begin{thm}\label{short}
	$\dim_H(\calg) \leq 1.7415$.
\end{thm}

The Rauzy Gasket has a number of interesting recent applications. Gamburd, Magee and Ronan \cite{gmr} showed asymptotic estimates for integer solutions of the Markov-Hurwitz equations featuring $\dim_H(\calg)$.
Hubert and Paris-Romaskevich in \cite{hp} considered triangular tiling billiards, modelling refraction in crystals. The gasket $\calg$ parameterises triangles admitting trajectories which escape non-linearly to infinity and closed orbits which approximate fractal-like sets.

In section 2 we give the technical result which leads to the bound in Theorem \ref{short}.  This is formulated in terms of certain infinite matrices.
In section 3 we give elementary preliminary bounds on the area and diameter of small triangles given as the images of $\De$ under compositions of the maps $T_1$, $T_2$ and $T_3$.
In section 4 we use these to obtain a bound for the dimension, provided an associated sequence of real numbers $X_n$ converges to zero.
In sections 5 and 6 we present the core of the proof. In section 5, we use the estimates from section 3 to bound $X_n$ in terms of expressions satisfying an iterative relation.
In section 6 we use the renewal theorem to deduce that $X_n \to 0$ under the hypotheses of Theorem \ref{long}. Finally, in section 7 we apply Theorem \ref{long} empirically to deduce the bound in Theorem \ref{short}.

A fuller account appears in \cite{thesis}.

\section{A formal statement}

The bound in Theorem \ref{short} is a special case of a decreasing sequence of upper bounds, indexed by a parameter $m \in \NN_{\geq 2}$. Each bound can be described using powers of an infinite matrix.
\begin{definition}[An index set]
	Let $\calv = \bigcup_{k=1}^{m-1}\calv_k$ denote the finite set where, for $k<m$,
		$$
				\calv_k
			:=
				\{1\}^k \times \{2\} \times \{1,2,3\}^{m-k}
			=
				\left\{
					(
					1^{k}, 2, v_{k+2}, \ldots, v_{m+1}
					)
				\ :\ %
					v_{k+2}, \ldots, v_{m+1} \in \{1,2,3\}
				\right\}
				,
		$$
	where we denote, e.g., $1^k = \overbrace{1, \ldots, 1}^{\smash{k}}$%
	.
 	i.e.,  $\calv$ is the family of strings of length $m+1$  beginning with a sequence of  $1$s of length $k$, followed by a $2$, then a sequence of $1$s,  $2$s and $3$s of length  $m-k$.
\end{definition}

\begin{rem}\label{sim}
	The elements of $\calv$ represent orbits in $\{1,2,3\}^{m+1}$ under the natural action of the dihedral group, excluding the orbits of $(1^m,2)$ and $(1^{m+1})$.

	More specifically, any two words
		$\un i, \un v\in\{1,2,3\}^{m+1}$ will be considered equivalent (written $\un i \sim \un v$)
	if there is some permutation $\pi$ of $\{1,2,3\}$ such that $\pi(i_j) = v_j$ for all $j=1,\ldots,m+1$.
	For example, if $i_1\neq i_2$, $(i_1,i_2^{m})$ is equivalent to $(1,2^{m})=: \ostar$, which we consider as a distinguished element of $\calv$.

 	For each $\un v \in \calv$ and $n \in \NN$, there will be $2 \times 3^{n-m}$ strings $\un i \in \{1,2,3\}^n$ for which the truncation  $(i_1, \ldots, i_{m+1}) \in \{1,2,3\}^{m+1}$ is equivalent to  $\un v$.

\end{rem}

For  $n \geq 1$ and $\un i = (i_1, \ldots, i_n) \in \{1, 2, 3\}^n$ we denote the composition $T_{\un i} = T_{i_1}\circ  \cdots  \circ T_{i_n}$ and its image
$\De_{\un i} =  T_{\un i}(\De)  \subset \De$.
The following is easily seen \cite{thesis}.

	\begin{lemma}
		Each $\De_{\un i}$,  $\un i \in \{1, 2, 3\}^n$, $n \geq 1$ is again a triangle.
	\end{lemma}

As we shall see later, these $3^n$ small triangles provide a useful family of covers for $\calg$.

We now formally define a finite matrix $B$, whose size depends on $m$ and whose entries depend on $\de$, as follows. Since $B$ acts as a weighted adjacency matrix, we first need a condition for adjacency.

\begin{definition}[Adjacency]
	Given $\un v = (1^k,2,v_{k+2},\ldots,v_{m+1}) \in \calv$, we write $\un v \mapsto_1 \un v'$, where $\un v' \in \calv \cup\{m\}$ is as follows:
	\begin{enumerate}
		\item[a)]
			If $k \leq m-2$, then $\un v' = (1^{k+1},2,v_{k+2}, \ldots, v_{m})$.
		\item[b)]
			If $k = m-1$, then  $\un v' = m$.
    \end{enumerate}
	Moreover, for $j = 2, 3$, we write $\un v \mapsto_j \un v'$, where $\un v' = (1,2^k,\ldots) \in \calv_1$ is equivalent to $(j,1^k,2,v_{k+2}, \ldots, v_{m})$.
\end{definition}

This definition describes the non-zero indices of the matrix $B$, as follows.

\begin{definition}[A finite matrix]
    Fix the value $\la := \frac{3}{2} - \frac{1}{\sqrt{3}} = 0.9226 \ldots$ and $\de>0$. Then, we can consider the square matrix $B$ indexed by $\calv \cup \{m\}$, defined by
		$$
				B_{\un v,\un w}
			=
				\begin{cases}
					\displaystyle
					\max_{
					    x \in \De_{\un w}
					    }
 						(2-x_j)^{-3\de - \la (1-\de)},
				&
					\hbox{ if } \un w \mapsto_j \un v;
				\\
							0, &\hbox{ otherwise}.
				\end{cases}
		$$
	(Note that $B_{\ostar, \un v} = B_{\un v,m} = 0$ for all $\un v\in \calv\cup \{m\}$.)
\end{definition}

\begin{definition}
To define the infinite matrix, we introduce the following values, for $k \in \NN$.
	$$
			a_k
		=
			\left(
				\frac{k+1}{2k+1}
			\right)^{3\de + \la (1-\de)}
 		+
 			\frac{1}{8^\de}
 			\left(
 				\frac{k+1}{2k+1}
 			\right)^{\la (1-\de)}
 		\qquad
 		\hbox{and}
 		\qquad
			b_k
   		=
   			\left(
   				\frac{k+2}{k+3}
   			\right)^{3\de + \la (1-\de)},
 	$$
\end{definition}

We now extend the finite matrix $B$ to give an infinite matrix.

\begin{definition}[An infinite matrix]
	Fixing $\de > 0$, we define the infinite matrix $D$ indexed by $\calv \cup\{m, m+1, \ldots\}$ as follows.%
	\footnote{Here, $N-1$ is the cardinality of $\calv$, and in the ordering of $\calv\cup\{m\}$, we take $\ostar$ to be first and $m$ last.}
		\begin{align*}
				D
		&=
			\begin{pmatrix}
                   B_{1,1}&\cdots&B_{1,{N-1}}&B_{1, N} + a_m&a_{m+1}&a_{m+2}& \cdots
				\\
					B_{2,1}&\cdots&B_{2, {N-1}}&B_{2,N}&0&0& \cdots
				\\
					\vdots&\ddots&\vdots&\vdots&\vdots &\vdots & \cdots
				\\
					B_{{N},1}&\cdots&B_{{N},{N-1}}&B_{N,N}&0&0 & \cdots
				\\
					0&\cdots&0&b_m&0&0& \cdots
				\\
					0&\cdots&0&0&b_{m+1}&0& \cdots
				\\
					0&\cdots&0&0&0&b_{m+2}& \cdots
				\\
					\vdots&\vdots&\vdots&\vdots&\vdots&\vdots& \ddots
			\end{pmatrix}
		\\
		&=
			\begin{pmatrix}
					0&\cdots&0&a_m&a_{m+1}&a_{m+2}& \cdots
				\\
					B_{2,1}&\cdots&B_{2, {N-1}} &0&0&0& \cdots
				\\
					\vdots&\ddots&\vdots&\vdots&\vdots &\vdots & \cdots
				\\
					B_{{N},1}&\cdots&B_{{N},{N-1}}&0&0&0 & \cdots
				\\
					0&\cdots&0&b_m&0&0& \cdots
				\\
					0&\cdots&0&0&b_{m+1}&0& \cdots
				\\
					0&\cdots&0&0&0&b_{m+2}& \cdots
				\\
					\vdots&\vdots&\vdots&\vdots&\vdots&\vdots& \ddots
			\end{pmatrix}.
	    \end{align*}
\end{definition}

Let $(D^k)_{i,j}$ be the entry corresponding to the $(i,j)$th entry of the $k$th power of $D$. Our main technical result is the following, which we will rephrase later in a more explicit way.

\begin{thm}\label{long}
    If $\de > \frac {1-\la}{3-\la} = 0.037\ldots$ satisfies $\sum_{k=1}^\infty (D^k)_{1,1}\leq 1$, then $\dim_H(\calg) \leq 1 + \de$.
\end{thm}

In particular, letting $m=9$ and estimating numerically the value of $\de \approx 0.7415$ (rounding up to four decimal places) giving equality in the hypothesis of Theorem \ref{long} gives the bound in Theorem \ref{short} in the introduction as a corollary (see \S 7).

\section{Triangle estimates}

In this section we collect together elementary but useful estimates for the triangles $\De_{\underline i}$.

\begin{lemma}[Area estimate]\label{area}
    If $\underline i = (i_1, \ldots, i_n) \in \{1,2,3\}^n$ and $j \in \{1,2,3\}$
    then, writing  $j \underline i = (j, i_1, \ldots, i_n)$, we have that
        $$
                \frac
                    {\area(\De_{j \underline i})}
                    {\area(\De_{\underline i})}
            \leq
                \max_{x \in \De_{\underline i}}
                    (2-x_j)^{-3}.
        $$
\end{lemma}

\begin{proof}
    By a change of variables,
        $$
                \area (\De_{j\underline i})
            =
                \area (T_j\De_{\underline i})
            =
                \int_{\De_{\underline i}}
                    \Jac T_j(x)
                \;\mathrm dx
            \leq
                \max_{x \in \De_{\underline i}}
                    \big( \Jac {T}_j(x) \big)
                \area (\De_{\underline i}).
        $$
    To complete the proof, we now show that $\Jac T_j (x) = (2-x_j)^{-3}$. If $j=1$, with respect to the orthogonal basis
        $( \frac1{\sqrt2}(\frac\partial{\partial x_2}-\frac\partial{\partial x_3}),\frac1{\sqrt6}(2\frac\partial{\partial x_1}-\frac\partial{\partial x_2} -\frac\partial{\partial x_3}))$,
    the derivative map $DT_j(x)$ takes the following form:
        \begin{equation}
                DT_1(x)
            =
                \frac 1 {(2-x_1)^2}
                \left(
                    \begin{array}{cc}
                        1 & 0 \\
                        \frac1{\sqrt{3}}(
                        x_2-x_3)
                         & 2-x_1 \\
                    \end{array}
                \right).
        \label{eq-derivative as matrix}
        \end{equation}
    Thus, $\Jac T_1 (x) = (2-x_1)^{-3}$. The other cases ($j\neq 1$) follow by symmetry.
\end{proof}

\begin{rem}
    In fact, one can deduce a simple formula for $\area(\De_{\un i})$ from the precise form of the above jacobean, using matrix products (see \cite{thesis}).
\end{rem}

We can apply a similar reasoning to the associated diameters.

\begin{lemma}[Diameter estimate]\label{diameter}
If $\un i = (i_1, \ldots, i_n) \in \{1,2,3\}^n$ and $j \in \{1,2,3\}$ then
    $$
                \frac
                {
                    \diam(\De_{j \un i})
                }
                {
                    \diam(\De_{\un i})
                }
        \leq
            \max_{ x \in \De_{\un i}}
                (2-x_j)^{-\la}.
    $$
\end{lemma}
\color{black}
\begin{rem}
	The constant $\la$ used as the 
	%optimal
	 exponent in Lemma \ref{diameter} was suggested by noticing empirically that the contours of the operator norm $\|DT_1(x)\|$ of $DT_1(x)$ were very similar to those of its main eigenvalue, $(2-x_1)^{-1}$. Hence it was natural to consider the ratio
		$$
%				-\la
%			=
%			\sup_{x \in \De\setminus(1,0,0)}
				\frac
					{\log\|DT_1(x)\|}
					{\log(2-x_1)},
		$$
	whose supremal value is obtained in the limit as $x \to (1,0,0)$, and may be computed using L'H\^opital's rule.
\end{rem}
\color{black}
\begin{proof}
    Consider the operator norm, $\|DT_j(x)\|$, of the derivatives $T_j(x)$ (i.e., with respect to the ambient metric in $\RR^3$). Considering $j=1$, the matrix in \eqref{eq-derivative as matrix} satisfies
        $$
                \big\|DT_1(x)\big\|^2
            =
                \frac
                    1%{(2-x_1)^{-2}}
                    % {b\sqrt2}
                    {2b^2}
                \left(
                {%\sqrt{
                        1 + a + b
                    +
                        \sqrt{
                            (1+ a + b)^2
                        -
                            4b
                        }
                }
                \right)%^{\frac12}
                ,
        $$
    where $a = \frac13(x_2-x_3)^2$ and $b = (2-x_1)^2$. Using the simple bound $a \leq \frac13 (1-x_1)^2$, one has that, for all $x \in \De$,
        $$
                \big\|DT_1(x)\big\|^2
            \leq
                f(x_1)
            :=
                \frac{
                        % \sqrt{
                            2x_1^2 - 7x_1 + 8 +2(1-x_1)
                            \sqrt{
                                 x_1^2-5x_1+7
                            }
                        % }
                    }
                    {
                        3
                        (2-x_1)^4
                    }.
        $$

    We claim that $g:t\mapsto f(t)(2-t)^{2\la}$ is increasing on $[0,1]$, and hence has maximum $g(1) = 1$ (we defer the proof of this statement to the appendix), i.e., $\|DT_1(x)\| \leq \sqrt{f(x_1)} \leq (2-x_1)^{-\la}$. The result for $j=1$ follows simply from the mean value theorem and for $j\neq1$ by symmetry.
\end{proof}
\color{black}
\begin{rem}
	In the context of higher-dimensional Rauzy gaskets (see \cite{arnoux-starosta} for the definition), an effective version of Lemma \ref{diameter} certainly applies, since the derivative maps there have a very similar structure to that observed above.
\end{rem}
\color{black}

In view of the previous lemmas, it is useful to have more explicit bounds involving $(2-x_j)^{-1}$ depending on where the point $x \in \De$ approximately lies. We henceforth denote
    $$
            A_{n,k}
        =
            \{
                \un i \in \{1,2,3\}^n
            \;:\;
                i_1 = \cdots = i_k \neq i_{k+1}
            \}
            \qquad(1\leq k < n)
    $$
and work with the following simple bounds.

\begin{lemma}
    If $\un i \in A_{n,k}$ and $i_1=i,j,l \in \{1,2,3\}$ are distinct
    then, for any $\de>0$,
        \begin{gather*}
                \max_{x \in  \De_{\underline i}}
                    (2-x_i)^{-1}
            \leq
                % \left(
                    \frac
                        {k+2}
                        {k+3}
                % \right)
                ,
           \\
           \phantom{
            \qquad
                \hbox{and}}
           % \qquad
                \max_{x \in  \De_{\underline i}}
                    (2-x_j)^{-1}
            \leq
                % \left(
                    \frac
                        {k+1}
                        {2k+1}
                % \right)
                ,
            \qquad
                \hbox{and}
            \\
                \max_{x \in  \De_{\underline i}}
                    \left\{
                        (2-x_j)^{-3\de}
                    +
                        (2-x_l)^{-3\de}
                    \right\}
\leq   \left( \frac{k+1}{2k+1}\right)^{3\de} + \frac 1 {8^\de}.
    \end{gather*}
\end{lemma}

\begin{proof}
    By symmetry, it suffices to consider $\underline i = (i_1, \ldots, i_n)$ with $i_1 = \cdots = i_{k} = 1$ and $i_{k+1} \neq 1$. Then $\De_{\un i}$ is contained in
        \begin{equation}
        \label{R_k}
                \text{cl}\big(
                    T_1^k(\De) \setminus T_1^{k+1}(\De)
                \big)
            =
                \bigg\{
                    x \in \De
                    \;:\;
                        \frac
                            {k}
                            {k+1}
                    \leq
                        x_1
                    \leq
                        \frac
                            {k+1}
                            {k+2}
                \bigg\}
        \end{equation}
    (cl denoting the topological closure), where we have used that $T_1^k(\De) = \{x \in \De\,:\,x_1\geq \frac{k}{k+1}\}$. The first two bounds follow directly:
        $$
                (2-x_1)^{-1}
            \leq
                \left(
                    2-\frac{k+1}{k+2}
                \right)^{-1}
            =
                    \frac{k+2}{k+3}
            \qquad\text{and}
            \qquad
                 (2-x_j)^{-1}
            \leq
                \left(
                    2-\frac{1}{k+1}
                \right)^{-1}
            =
                    \frac{k+1}{2k+1}.
        $$

    For the third bound, we observe that the function $f(x) := (2-x_2)^{-3\de} + (2-x_3)^{-3\de}$ is convex on $\De$ and thus takes its maximum on $T_1^k(\De)$ at one of its vertices:
        $$
                \max_{\De_{\un i}}
                    f
            \leq
                \max_{T_1^k(\De)}
                    f
            =
                \max
                    \left\{
                        f\circ T_1^k(1,0,0),
                        f\circ T_1^k(0,1,0),
                        f\circ T_1^k(0,0,1)
                    \right\}.
        $$
    More explicitly, noting that $f\circ T_1^k(1,0,0) = f(1,0,0) \leq f\circ T_1^k(0,1,0) = f\left( \frac k {k+1}, \frac 1 {k+1},0 \right) = f\left(\frac k {k+1},0, \frac 1 {k+1} \right) =  f\circ T_1^k(0,0,1)$, one has%, for any $x \in \De_{\un i}$,
        $$
                \max_{\De_{\un i}}f
            \leq
                    \left(
                        2-\frac{1}{k+1}
                    \right)^{-3\de}
                +
                    2^{-3\de}
            =
                \left(\frac{k+1}{2k+1}\right)^{3\de} + \frac 1 {8^\de},
        $$
    as required.
\end{proof}

\section{Cover  estimates}

The upper bound on the dimension in Theorem \ref{long} is based on finding a value $\de \in (0,1)$ so that Lemma 4.1 below applies.

Its proof is simple and based on a simple sequence of open covers of $\calg$, $\mathcal U_n$, each obtained by covering the set of $n$th level triangles, $\{\De_{\un i} \hbox{ : } |\un i| = n\}$.

\begin{lemma}\label{dimbound}
    Assume $\de>0$ and that the sequence
        $$
                X_n
            :=
                \sum_{|\un i| = n}
                    \area(\De_{\un i})^{\de}
                    \diam(\De_{\un i})^{1-\de}
            \to
                0
            \qquad\hbox{as}
            \qquad
                n \to \infty.
        $$
    Then $\dim_H(\calg) \leq 1 + \de.$
\end{lemma}

		\begin{figure}[hbt]
			\begin{tikzpicture}

			\foreach \x/\y in {5/2}
			{
				\draw[fill = yellow] (0,0) -- (0,\y)
						node[midway,anchor = east]{$\displaystyle  \frac{2\area(\De_{\un i})}{{\diam(\De_{\un i})}}$}
					-- (\x,\y) -- (\x,0);

				\draw[thick, fill = white] (0,0) -- (\x/5,\y) -- (\x,0) -- (0,0) node[midway,anchor=north]{$\diam(\De_{\un i})$};

				\draw[dotted] (\x+2,0.5*\y) -- (\x+2, 0.5*\y);

				\draw[fill = black, opacity = 0.3] (\x+1.5,0) -- (\x+1.5,\y) -- (2*\x+1.5,\y) -- (2*\x+1.5,0) -- (\x+1.5,0);

				\draw[fill = yellow] (\x+1.5,0) -- (\x+1.5,\y) -- (2*\x+1.5,\y) -- (2*\x+1.5,0) -- (\x+1.5,0);

				\foreach \n in {0,1,2,...,4}{

					\draw[fill=black, opacity = 0.1] (\x+2+\n,0.5*\y) circle (1.41423);
				}

				\draw(\x+1.5,0) -- (\x+1.5,\y) -- (2*\x+1.5,\y) -- (2*\x+1.5,0) node[midway,anchor = west]{$a$} -- (\x+1.5,0) node[midway, anchor = north]{$b$};

				\foreach \n in {0,1,2,...,4}{
					\draw[fill=black] (\x+2+\n,0.5*\y) circle (0.05);
				}
			}

			\draw (4/3,1) node {$\De_{\un i}$};

			\end{tikzpicture}
		\caption[Covering a triangle by disks]{Left: covering $\De_{\un i}$ by a rectangle. Right: covering a rectangle by open disks.}
		\label{figr-cover rectangles}
		\end{figure}
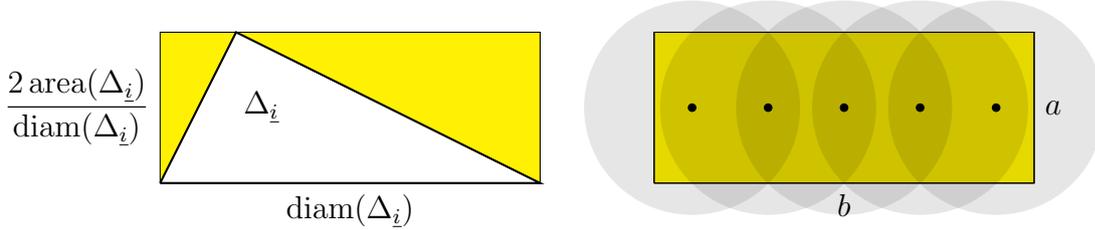

\begin{proof}
    As illustrated in Figure 2, each triangle $\De_{\un i}$ is contained in a rectangle with side lengths $b = \diam(\De_{\un i})$ and $a = 2  \area(\De_{\un i})/ \diam(\De_{\un i})$, which  can in turn be covered by $[2b/a]$ disks of diameter $2a$.  That is, $\De_{\un i}$ can be covered by $\diam(\De_{\un i})^2/ \area(\De_{\un i})$ disks of diameter $4\area(\De_{\un i})/ \diam(\De_{\un i})$.
    Denoting by $\mathcal U_n$ the union of the disks covering all of the triangles $\De_{\un i}$, this gives
        $$
                \sum_{U \in \mathcal U_n}
                    \diam(U)^{1+\de}
            \leq
                4^{1+\de}
                    \sum_{|\un i|=n}
                \left(
                    \frac
                        { \diam(\De_{\underline i})^2}
                        { \area(\De_{\underline i})}
                \right)
                \left(
                    \frac{
                        \area(\De_{\underline i})}
                        {\diam(\De_{\underline i})}
                \right)^{1+\de}
            =
                4^{1+\de} X_n.
        $$
    The result follows from the standard definition of the Hausdorff dimension (see, e.g., \cite{falconer}).
\end{proof}

\begin{rem}
    Perhaps surprisingly, this lemma appears to give a significant improvement on working exclusively with either $\diam(\De_{\underline i})$  or $\area(\De_{\underline i})$ (as in \cite{fou} and \cite{ar}).
\end{rem}

\smallskip
To apply Lemma \ref{dimbound} we want to bound $X_n$ from above using a partition of
$\{1,2,3\}^n$.
We then apply different bounds to the triangles $T_j(\De_{\un i})$ according to which element of the partition the index $\un i$ lies in.

\begin{definition}[Partitioning up the sequences]
    For each $n > m$, we can partition
        $$
                \{1,2,3\}^n
            :=
                \bigcup_{\un v \in \calv}
                    A_{n,\un v}
            \cup
                \bigcup_{k=m}^{n-1}
                    A_{n,k}
            \cup
                \bigcup_{j=1}^3
                     \{(j^n)\},
        $$
    where $A_{n,k}$ is as above and $A_{n,\un v} = \{\un i \in \{1,2,3\}^n\;:\; (i_1,\ldots, i_{m+1}) \sim \underline v\}$, recalling the equivalence relation $\sim$ on page 3. %\pageref{sim}.

\end{definition}

This partition naturally gives the following components of $X_n$, for $n>m$.

\begin{definition}
    For $n > m$ and any $\alf \in \calv \cup \{m,\dots,n-1\}$, we write
        $$
                X_{n,\alf}
            :=
                \sum_{\un i \in A_{n,\alf}}
                    \area(\De_{\un i})^\de
                    \diam(\De_{\un i})^{1-\de}.
        $$
\end{definition}

We now derive bounds on $X_{n+1}$ using estimates on $X_{n,\un v}$ and $ X_{n,k}$.

\section[Inductive bounds]{Inductive bounds on $X_{n,\un v}$ and  $X_{n,k}$}

We now turn to the problem of showing that $X_n \to 0$ as $n \to \infty$, which will enable us to apply Lemma \ref{dimbound}.
Our method, broadly speaking, is based on obtaining inductive bounds bounding

	\begin{enumerate}
 		\item
			$X_{n+1,k+1}$ in terms of $X_{n,k}$, and
		\item
 			$X_{n+1,\un v}$ in terms of $X_{n,\un w}$ and $X_{n,k}$.
	\end{enumerate}
These will prove useful in applying the renewal theorem in the next section.

\subsection{Bounds of terms indexed by numbers}

% The next lemma suffices to obtain Theorem \ref{long} when $m=1$; since we have assumed $m\geq2$, we also need to combine it with Lemma 5.2.

The next simple lemma is independent of $m$, and features the constants $b_k$.

\begin{lemma}[Number Lemma]\label{number}
    Suppose $\de \in (0,1)$. Then, for all $n > k$,
         \begin{gather*}
                X_{n+1,k+1}
            \leq
                b_k X_{n,k}
                % ;
            % \\
            %     X_{n+1,1}
            % \leq
            %     c_n
            % +
            %     \sum_{j=1}^{n-1}
            %         a_j X_{n,j}
            .
        \end{gather*}
\end{lemma}

\begin{proof}
    % For the first inequality, s
    Since $j\un i \in A_{n+1,k+1}$ if and only if  $\un i \in A_{n,k}$ and $j = i_1$, we may write
        \begin{align*}
                X_{n+1,k+1}
            &=
                \sum_{\un i \in A_{n,k}}
                    \area(\De_{i_1\un i})^\de
                    \diam(\De_{i_1\un i})^{1-\de}
            \\
            &\leq
                \sum_{\un i \in A_{n,k}}
                    \left(
                        \frac
                            {k+2}
                            {k+3}
                    \right)^{3\de}
                    \left(
                        \frac
                            {k+2}
                            {k+3}
                    \right)^{\la(1-\de)}
                    \area(\De_{\un i})^\de
                    \diam(\De_{\un i})^{1-\de}
            \\
            &=
                b_k X_{n,k},
        \end{align*}
    where we have applied Lemmas \ref{area} and \ref{diameter}.
\end{proof}

\subsection{Bounds of terms indexed by words}

We can bound $X_{n+1,\un v}$, $n > m,$ with $\un v \in \calv\cup\{m\}$, using the matrix $B$ and the constants $a_k$, as formulated in the next lemma.

\begin{lemma}[Word Lemma]\label{word}
    Let $\de\in (0,1)$, $n > m$ and $\un v \in \calv \cup \{m\}$. Then, recalling $\ostar = (1,{2^m})$,  there exists a positive sequence $c_n$ such that $\sum_n c_n < \infty$ and
        \begin{align*}
                X_{n+1,\ostar}
            &\leq
                c_n
            +
                \sum_{k=m}^{n-1}
                    a_k
                    X_{n,k}
            \qquad
                \hbox{and}
            \\
                X_{n+1,\un v}
            &\leq
                \sum_{\un w \in \calv}
                    B_{\un v, \un w}
                    X_{n,\un w}
            \qquad
            \hbox{if}
            \ %
            \un v \neq \ostar.
        \end{align*}
\end{lemma}

\begin{proof}
    For the first inequality, note that $j\un i \in A_{n+1,\ostar}$ if and only if $j\neq i_1$ and $\un i \in \bigcup_{k=m}^{n-1} A_{n,k} \cup \{(1^n),(2^n),(3^n)\}$. Therefore, applying  Lemmas \ref{diameter} and \ref{area} in turn and using symmetry, one has that
        \begin{align*}
                X_{n+1,\ostar}
            &=
                6
                \area(\De_{(2,1^n)})^\de
                \diam(\De_{(2,1^n)})^{1-\de}
            +
                \sum_{k=m}^{n-1}
                \sum_{\un i \in A_{n,k}}
                \sum_{j \neq i_1}
                    \area(\De_{j\un i})^\de
                    \diam(\De_{j\un i})^{1-\de}
            \\
            &\leq
                c_n
            +
                \sum_{k=m}^{n-1}
                \sum_{\un i \in A_{n,k}}
	                \left(
	                     \frac
	                         {k+1}
	                         {2k+1}
	                \right)^{\la(1-\de)}
                	\diam(\De_{\un i})^{1-\de}
                \sum_{j \neq i_1}
                    \area(\De_{j\un i})^\de
            \\
            &\leq
                c_n
            +
                \sum_{k=m}^{n-1}
                \sum_{\un i \in A_{n,k}}
                    \left(
                        \frac
                            {k+1}
                            {2k+1}
                    \right)^{\la(1-\de)}
                    \diam(\De_{\un i})^{1-\de}
                    \left(
                        \left(
                            \frac
                                {k+1}
                                {2k+1}
                        \right)^{3\de}
                    +
                        \frac 1 {8^\de}
                    \right)
                    \area(\De_{\un i})^\de
            \\
            &=
                c_n
            +
                \sum_{k=m}^{n-1}
                    a_k X_{n,k},
        \end{align*}
    where $c_n := 6 \area(\De_{(2,1^n)})^\de \diam(\De_{(2,1^n)})^{1-\de}$, as required.

    Regarding the second inequality: The combinatorics above and in the previous proof show that $\un i \in \bigcup_{k=m}^{n-1} A_{n,k} \cup \{(1^n),(2^n),(3^n)\}$ implies $j\un i \not \in A_{n+1,\un v}$ for every $\un v \in \calv \cup\{m\} \setminus \{\ostar\}$. Fixing such a $\un v$ and considering the contrapositive, one sees that $j\un i \in A_{n+1,\un v}$ implies that $\un i \in A_{n,\un w}$ and $\un w \mapsto_{j'}\un v$ for some $\un w, j'$. Assuming $(i_1,\ldots, i_{m+1}) = \un w$ for simplicity, we have that $j' = j$ and $\De_{\un i}\subset \De_{\un w}$. Hence, by Lemmas \ref{area} and \ref{diameter},
        $$
                \frac
                    {
                     \area(\De_{j\un i})^\de
                     \diam(\De_{j\un i})^{1-\de}
                    }
                    {
                     \area(\De_{\un i})^\de
                     \diam(\De_{\un i})^{1-\de}
                    }
            \leq
                \max_{\un x \in \De_{\un w}}
                    (2-x_j)^{-3\de - \la(1-\de)}
            =:
                B_{\un v,\un w}.
        $$
    Thus, we may write the following.
        \begin{align*}
                X_{n+1,\un v}
            =%&=
                \sum_{j=1}^3
                    \sum_{\substack{\un w \in \calv:\\[2pt]\un w \mapsto_j \un v}}
                        \sum_{\un i \in A_{n,\un w}}
                            \area(\De_{j\un i})^\de
                            \diam(\De_{j\un i})^{1-\de}
            % \\
            &\leq
                \sum_{\substack{\un w \in \calv:\\B_{\un v,\un w}\neq 0}}
                    \sum_{\un i \in A_{n,\un w}}
                        B_{\un v,\un w}
                        \area(\De_{\un i})^\de
                        \diam(\De_{\un i})^{1-\de}
            \\
            &=
                % \phantom{\sum_{\un w \mapsto_j}{}}
                \sum_{\un w \in \calv}
                    B_{\un v,\un w}
                    X_{n, \un w},
        \end{align*}
    as required.

    Finally, a direct computation of the vertices of $\De_{(2,1^n)} = T_2T_1^n(\De)$ shows that its diameter and area are proportional to $n^{-1}$ and $n^{-2}$ respectively, as $n \to \infty$. Hence $c_n$ is proportional to $n^{-1-\de}$ as $n\to \infty$, and is thus summable for every $\de>0$, completing the proof.
\end{proof}

            \section{Renewal Theorem}

We now use the iterative bounds of the last section to show that $X_n \to 0$ as $n \to \infty$, under the hypotheses of Theorem \ref{long}. This uses the following mild adaptation of the classical renewal theorem of Feller \cite[p.330]{feller}.

\begin{thm}
	Suppose that the sequences $(Y_n)_{n=0}^\infty$,
	$(\mu_n)_{n=1}^\infty$ and $(\nu_n)_{n=1}^\infty$ are non-negative and satisfy $\sum_{n=1}^\infty \mu_n < 1$, $\sum_{n=1}^\infty \nu_n < \infty$ and
		$$
			Y_n \leq \nu_n + \sum_{k=1}^n \mu_k Y_{n-k}
		$$
	for each $n\in \NN$. Then $Y_n \to 0$ as $n \to \infty$.
\end{thm}

We apply this theorem with $Y_n = X_{n+m+1,\ostar}$ to give Theorem \ref{long}, which can be rewritten as follows.
\begin{thm}
    If $\de > \frac {1-\la}{3-\la} = 0.037\ldots$ satisfies
        \begin{equation}
                \sum_{k=1}^\infty
                    (B^k)_{m,\ostar}
                \color{black}
                \cdot
                \color{black}
                \sum_{k=m}^\infty
                    a_k
                    \prod_{i=m}^{k-1}
                        b_i
            \leq
                1,
        \label{cond}
        \end{equation}
    then $\dim_H(\calg) \leq 1 + \de$.
\end{thm}

\begin{proof} Assume that \eqref{cond} holds with a strict inequality (the case of equality follows in the limit, since the LHS is decreasing in $\de$), and that $\de<1$ (the conclusion being otherwise trivial). Applying Lemma \ref{number} repeatedly in the first estimate of Lemma \ref{word} gives
    $$
            X_{n+1,\ostar}
        \leq
            c_n
        +
            \sum_{k=m}^{n-1}
                a_k
                \prod_{i=m}^{k-1}
                    b_i
            X_{n+m-k,m}.
    $$
Moreover, the second estimate of Lemma \ref{word} extends inductively to give, for any  $k \in \NN,$
    $$
            X_{n+m+k,m}
        \leq
            \sum^{n+m-1}_{j=1}
                (B^j)_{m,\ostar}
                X_{n+m,\ostar}
        +
            \sum_{\substack{\un v \in \calv:\\\un v \neq \ostar}}
                (B^{n+k-1})_{m,\un v}
                X_{m+1,\un v}.
    $$

 Putting these two together gives the renewal-style inequality
        $$
                X_{n+1, \ostar}
            \leq
                \nu_n
            +
                \sum_{k=1}^{n-m}
                    \mu_k
                    X_{n+1-k,\ostar}
            ,
        $$
    with coefficients
        $$
                \mu_k
            =
                \sum_{i+j=k}
                    (B^i)_{m,\ostar }
                    a_{m+j}
                    \prod_{l=m}^{m+j-1}
                        b_l
        $$
    and remainder term
        $$
                \nu_n
            =
                    c_n
                +
                    \sum_{k=m}^{n-1}
                        a_k
                        \prod_{i=m}^{k-1}
                            b_i
                 \color{black}
                 \cdot
                 \color{black}                          
                    \sum_{\un v \in \calv\setminus\{\ostar\}}
                        (B^{n-k-1})_{m,\un v}
                        X_{m+1,\un v}.
        $$

    Consider the other hypotheses of the renewal theorem. The hypothesis on the coefficients is precisely our assumption that strengthens \eqref{cond}:
        $$
                \sum_{k=1}^\infty
                    \mu_k
            =
                \sum_{k=0}^\infty
                    (B^k)_{m, \ostar}
	             \color{black}
	             \cdot
	             \color{black}
                \sum_{k=m}^\infty
                    a_k
                \prod_{i=m}^{k-1}
                   b_i
            <
                1.
        $$
    Regarding the summability of the remainder terms $\nu_n$: By elementary combinatorics, $B^{m+1}$ has only one zero row and zero column, so this inequality also ensures that the spectral radius $\rho = \rho(B) < 1$. In any case, by the Perron-Frobenius theorem there exists $C>0$ such that
        $$
                \sum_{n=1}^\infty
                    \nu_n-c_n
            \leq
                C
                \sum_{n=1}^{\infty}
                \sum_{k=m}^{n-1}
                    a_k
                    \prod_{i=m}^{k-1}
                        b_i
                \rho^{n-k}
            =
                C
                    \sum_{k=m}^{\infty}
                        a_k
                        \prod_{i=m}^{k-1}
                            b_i
                    \sum_{n=1}^{\infty}
                        \rho^n
            .%< \infty,
        $$
    Thus, since $\rho \in(0,1)$ and
        $$
                a_k
                \prod_{i=m}^{k-1}
                    b_i
            =
                \mathcal O\big(k^{-3\de-\la(1-\de)}\big)
                \qquad (k\to\infty),
        $$
    we see that the right hand side is summable if $\de > \frac{1-\la}{3-\la}$. That is, $\sum_n \nu_n < \infty$, and the renewal theorem gives $X_{n, \ostar} \to 0$ as $n \to \infty$.

    Finally, we use this to bound $X_n$ via the following lemma.
\begin{lemma}
    Given $m \in \NN$ and $\de\in(0,1)$, there exists $C > 0$ such that $X_n \leq  C X_{n+m+1, \ostar}$.
\end{lemma}

\begin{proof}
    For each $\un i \in\{1,2,3\}^n$, the word $\un i' = (1,2^m,i_1,\ldots,i_n)\in A_{n+m+1,\ostar}$. This word satisfies
        $$
                \frac
                    {\area(\De_{\un i'})^\de
                     \diam(\De_{\un i'})^{1-\de}}
                    {\area(\De_{\un i})^\de
                     \diam(\De_{\un i})^{1-\de}}
            =
                \frac
                    {\area(T_1T_2^m\De_{\un i})^\de
                     \diam(T_1T_2^m\De_{\un i})^{1-\de}}
                    {\area(\De_{\un i})^\de
                     \diam(\De_{\un i})^{1-\de}}
            \geq
                K^{(1+\de)(m+1)},
        $$
    where $K$ is obtained by bounding the \textit{minimum} singular value of $DT_j(x)$ uniformly over $x \in \De$: namely, applying the estimates of section 3,
        $$
                \big\|
                    D(T_j^{-1})
                    \big(T_j(x)\big)
                \big\|
            =
                \frac
                    {\Jac T_j(x)}
                    {\|DT_j(x)\|}
            \geq
                (2-x_j)^{\la-3}
            \geq
                2^{\la-3}
            =:
                K.
        $$
    In particular, with $C = K^{(1+\de)(m+1)}$
        $$
                X_{n+m+1,\ostar}
            \geq
                \sum_{|\un i| = n}
                    \area(\De_{\un i'})^\de
                     \diam(\De_{\un i'})^{1-\de}
            \geq
                C
                X_n,
        $$
    as required.
\end{proof}

Thus, $X_{n, \ostar} \to 0$ as $n \to \infty$ implies that $X_{n} \to 0$ as $n \to \infty$, and applying Lemma \ref{dimbound} completes the proof of Theorem \ref{long}.
\end{proof}

\section{Numerical estimates}

%  \newpage
  \begin{table}[htb]
	\begin{center}
%   		\begin{tabular}{l|c@{\qquad}lc}
% 			$m$&$\de_m$\\\cline{1-2}
% 			1&0.8512\\
% 			2&0.8285\\
% 			3&0.7978\\
% 			4&0.7764\\
% 			5&0.7624\\
% 			6&0.7534\\
% 			7&0.7475\\
% 			8&0.7435\\
% 			9&0.7407\\
% 		\end{tabular}
  		\begin{tabular}{c|cc@{\qquad\qquad}rc|c}
% 			$m$ &   $\de_m$ &&  $\ m$ &   $\de_m$  \\\cline{1-2} \cline{4-5}
% 			1   &   0.8512  &&   6   &   0.7534\\
% 			2&0.8285&&7&0.7475\\
% 			3&0.7978&&8&0.7435\\
% 			4&0.7764&&9&0.7407\\
% 			5&0.7624\\ { 0.7545, 0.7485, 0.7444,0.7415}
			$m$ &   $\de_m +1$ &&&$m$ &   $\de_m +1$  \\\cline{1-2} \cline{5-6}
			\vphantom{\Huge K}2&1.8285&&&6&1.7545\\
			3&1.7982&&&7&1.7485\\
			4&1.7771&&&8&1.7444\\
			5&1.7635&&&9&1.7415\\
		\end{tabular}
	\end{center}
		\caption{Upper bounds $\de_m+1$ on $\dim_H(\calg)$ for different choices of $m$, rounded upwards to four decimal places.}
	\end{table}
%  \bigskip

%For the precise code and an explanation thereof, see \cite{github} and the documentation therein. This code is comprised by simple functions, which we used with a manual Newton's method to find the values presented in Table 1, correct to four decimal places.

\color{black}

The usefulness of Theorem \ref{long} lies in our ability to check \eqref{cond} for given values $\de$ and $m$. Since the left hand side of \eqref{cond} is decreasing and continuous in each of its two factors, the infimal such $\de$  for which the theorem holds is the unique value $\de_m$ yielding equality.

The values of these $\de_m$ for $m$ up to 9 are given in Table 1. To check the condition for given $\de$ and $m$, we used Wolfram Mathematica to codify $B$ as a matrix and to compute the first factor of \eqref{cond},
	$$
			\sum_{k=0}^\infty
				(B^k)_{m, \ostar}
		=
			\left(
				\big(
					I-B
				\big)^{-1} 
			\right)_{m,\ostar},
	$$
as a ratio of two determinants (via Cramer's Rule); and the second factor,
	\begin{align*}
			\sum_{k=m}^\infty
				a_k
			\prod_{i=m}^{k-1}
				b_i
		&=
			\sum_{k=m}^\infty
				\left[\left(
					\frac
						{k+1}
						{2k+1}
				\right)^{3\de + \la (1-\de)}
	 		+
	 			\frac{1}{8^\de}
	 			\left(
	 				\frac
	 					{k+1}
	 					{2k+1}
	 			\right)^{\la (1-\de)}
	 		\right]
			\prod_{i=m}^{k-1}
	   			\left(
	   				\frac
	   					{i+2}
	   					{i+3}
	   			\right)^{3\de + \la (1-\de)}
	   	\\
	   	&=
%	   		(m+2)^{3\de + \la(1-\de)}
	   		\sum_{k=m}^\infty
				\left(
					\frac
						{(k+1)(m+2)}
						{(2k+1)(k+2)}
				\right)^{3\de + \la (1-\de)}
	 		+
	 			\frac{1}{8^\de}
	 			\left(
	 				\frac
	 					{(k+1)(m+2)}
	 					{(2k+1)(k+2)}
	 			\right)^{\la (1-\de)}
				(k+2)^{-3\de},
	\end{align*}
using the inbuilt \verb|NSum| function. For the precise code, explanation and a breakdown of the confirmation, see \cite{github} and the documentation therein. In particular, we rely on the closed-source Wolfram Mathematica programme (with its inbuilt error bounds) to accurately
\begin{itemize}
	\item calculate the determinant of large matrices, and
	\item estimate the polynomially decaying sum.
\end{itemize}

Both calculations can be rigorised (see \cite{github} for details) and we encourage all efforts to do so. However, any changes we observed were too insignificant to affect the final values given here: accuracy to three decimal places in each factor appears to suffice.

That $\de_m$ is computed only up to $m=9$ owes to the exponential size of $B$, which places a commensurable computing burden on the hardware. We note that slightly larger choices of $m$ should be possible with more efficient compiled programming languages.
\color{black}
\begin{rem}
	The elementary nature of this method suggests a similar approach might be applied to other related examples of sets to give bounds on their dimension\color{black}; though likely with different limiting returns.
\end{rem}
\color{black}

\appendix

\section{Proof that $g$ is increasing.}

Here, we give an elementary but delicate proof of the following technical fact, used in the proof of Lemma \ref{diameter}:

\begin{lemma}
	The function $g:[0,1]\to \RR$,
		$$
				g(t) 
			:= 
				f(t)(2-t)^{2\la}
			=
				\frac 13
				(2-t)^{-1-2/\sqrt3}
				\left(
					2t^2 -7t + 8
					+
					2(1-t)
					\sqrt{
						t^2 - 5t + 7
					}
				\right) 
		$$
	is increasing. In particular, $g(t) \leq g(1) = 1$ for all $t$.
\end{lemma}

\begin{proof}
We simply show that $g' > 0$ on $(0,1)$. Writing $\alf(t) = \sqrt{t^2-5t+7} > 0$, we have
		$$
				g'(t)
			=
				\frac{h(t)}{\alf(t)}
				\cdot
				\frac{1}{9} \cdot
				(2-t)^{\color{black}-2 - 2/\sqrt{3}\color{black}}
				,
		$$
	where $h(t) \in \mathbb Q\left[t,\alf(t),\sqrt3\right]$ is given by
		\begin{multline*}
				h(t)
			:= 
				\left(6-4 \sqrt{3}\right) t^3
			+
				\left(\left(4 \sqrt{3}-6\right) \alf(t) +24 \sqrt{3}-39\right) t^2
			\\+
				\left(\left(24-14 \sqrt{3}\right) \alf(t) -48 \sqrt{3}+87\right) t
			+
				\left(16 \sqrt{3}-18\right) \alf(t)
			+
				28 \sqrt{3}-72.
		\end{multline*}
	It remains to show $h$ is positive, using Taylor approximants for $\alf$ based at 1. Namely, since
		$$
				\alf^{(3)}(t)
			= 
				\frac 98 
				\frac{5 - 2t}{\alf(t)^5},
			\qquad
				\alf^{(4)}(t)
			= 
				\frac9{16}
				\frac
					{16t^2 - 80t + 97}
					{\alf(t)^7}			
		$$
	are positive for all $t \in (0,1)$, the remainder term given by Taylor's Theorem for the degree two and degree three Taylor polynomials are negative and positive, respectively, hence these polynomials are upper and lower bounds, respectively, for $\alf$. More explicitly, for $t\in[0,1]$,
		$$
				\frac
					{t^3-t^2-25t+73}
					{16\sqrt3}
			\leq
				\alf(t) 
			\leq
 				\frac
 					{t^2 - 14t + 37}
					{8\sqrt3}.
		$$
	Applying these bounds everywhere in the definition of $h$ (according to the signs of the coefficients) reduces to the concrete polynomial bound
		$$
		\begin{aligned}
				h(t)
			&\geq
				\left(
					5\sqrt{3}+8
				+
					\left(2-\sqrt{3}\right)t^2(t+1)
				-
					3\left(\sqrt{3}+2\right)t
				\right)
					(1-t)^2.\cr
%					&=				
%\left( \left((8 + 5\sqrt{3}) + (2-\sqrt{3})t^2 + (3-\sqrt{3})t^3\right) - 3(2+\sqrt{3})t \right)(1-t^2) 
\end{aligned}
		$$
	It is easy to see that the right hand side is everywhere positive, hence so too are $h$ and $g'$.
\end{proof}

%%%%%%%%%%% Example 1%%%%%%%%%%%%%%%%%%%%%%%%%%

\end{document}